\documentclass[11pt]{amsart} 
\usepackage{amsmath, amssymb, amsfonts, amsbsy, amsthm, latexsym, stmaryrd, mathtools, fullpage, enumerate, tikz, standalone}
\usepackage{geometry}
\usepackage{graphicx,float}
\usepackage[affil-it]{authblk}
\usepackage[mathscr]{eucal}
\usepackage{wrapfig}
\usepackage{todonotes}
\usepackage{tikz-cd}
\usepackage[foot]{amsaddr}
\allowdisplaybreaks

\usepackage[mathscr]{eucal}
\usepackage{url}
\usepackage{shuffle}

\usepackage{amsmath}
\DeclareMathOperator{\sh}{\shuffle}
\DeclareMathOperator{\bsh}{\hat{\shuffle}}

\newtheorem{thm}{Theorem}[section]
\newtheorem*{thm*}{Theorem}

\newtheorem{lem}[thm]{Lemma}
\newtheorem{prop}[thm]{Proposition}
\newtheorem{cor}[thm]{Corollary}
\theoremstyle{definition}
\newtheorem{remark}[thm]{Remark}
\newtheorem{defn}[thm]{Definition}
\newtheorem{example}[thm]{Example}
\newtheorem{ex}[thm]{Example}

\numberwithin{equation}{section}

\def\decon{{\Delta_{\mathrm{decon}}}}

\def\Q{{\mathbb{Q}}}
\def\P{{\mathbb{P}}}

\def\z{{\zeta}}

\def\bg{{\mathfrak{bg}}}

\def\gmot{{\mathfrak{g}^\mathfrak{m}}}

\def\qpoly{{\Q\langle e_0,e_1\rangle}}

\DeclareMathOperator{\Sh}{Sh}

\def\bDelta{{\Delta_\mathfrak{bl}}}
\DeclareMathOperator{\ic}{I}
\DeclareMathOperator{\bstar}{\hat{\star}}

\def\proj{{\P^1\setminus\{0,1,\infty\}}}

\title{A generalisation of quasi-shuffle algebras and an application to multiple zeta values}
\author{Adam Keilthy}
\address{Chalmers University of Technology, Gothenburg}
\email{keilthy@chalmers.se}
\thanks{This work was completed during the author's stay at the Max Planck Institute for Mathematics in Bonn. We would also like to extend our sincerest thanks to the reviewer, for their insightful and careful feedback.}

\begin{document}
\maketitle

\begin{abstract}
A large family of relations among multiple zeta values may be described using the combinatorics of shuffle and quasi-shuffle algebras. While the structure of shuffle algebras have been well understood for some time now, quasi-shuffle algebras were only formally studied relatively recently. In particular, Hoffman \cite{hoffmanquasi1} gives a thorough discussion of the algebraic structure, including a choice of algebra basis, and applies his results to produce families of relations among multiple zeta values and their generalisations \cite{hoffmanquasi2}. In paper, Hirose and Sato established of relations coming from a new generalised shuffle structure, lifting a set of graded relations established by the author \cite{keilthyblock1} to genuine ungraded relations. In this paper, we define a commutative algebra structure on the space of non-commutative polynomials on a countable alphabet, generalising the shuffle-like structure of Hirose and Sato. We show that, over the rational numbers, this generalised quasi-shuffle algebra is isomorphic to the standard shuffle algebra, allowing us to reproduce most of Hoffman's results on quasi-shuffle algebras. We then apply these results to the case of multiple zeta values, reproducing several known families of results and establishing several new families.
\end{abstract}

\section{Introduction}
The study of multiple zeta values
$$\z(n_1,\ldots,n_r) := \sum_{1\leq k_1<k_2<\cdots<k_r}\frac{1}{k_1^{n_1}\ldots k_r^{n_r}},$$
defined for positive integers $n_1,\ldots,n_r$ with $n_r\geq 2$, and their relations dates back to Euler who showed $\z(1,2)=\z(3)$. However, multiple zeta values did not receive significant attention until the 1990s when they began to appear in a wide number of areas, including quantum field theory \cite{bkconj}, knot theory \cite{barnatan}, and the study of associators \cite{drinfeld}. Of particular interest is to describe all relations among multiple zeta values, multiple polylogarithms and other such generalisations. A common conjecture is that, up to some renormalisation, all relations among multiple zeta values are given by the double shuffle relations. Indeed, shuffle algebras and their generalisations play an important role in describing relations among multiple zeta values and their generalisations. For example, the double shuffle relations are given by the existence of two algebra homomorphisms
\begin{align}
(\qpoly,\sh) &\to \mathcal{Z}\\
(\Q\langle Y\rangle, \ast) &\to \mathcal{Z}
\end{align}
where $\mathcal{Z}$ is the $\Q$-algebra of multiple zeta values. 

 The first of these homomorphisms - the shuffle relations - is a map from the shuffle algebra on the space of noncommutative polynomials in $\{e_0,e_1\}$ with product given recursively by
\begin{align*}
u\sh 1 &= 1\sh u = u,\\
e_iu\sh e_j v &= e_i(u\sh e_jv) + e_j(e_iu\sh v),
\end{align*}
for any monomials $u,v$ in $\{e_0,e_1\}$.

The second of these homomorphisms - the stuffle or harmonic relations - is a map from a quasi-shuffle algebra \cite{hoffmanquasi1,hoffmanquasi2} on the space of noncommutative polynomials in $Y=\{y_1,y_2,\ldots\}$ with product given recursively by
\begin{align*}
u \ast 1 &= 1\ast u = u,\\
y_ku\ast y_lv &= y_k(u\ast y_lv) + y_l(y_k\ast v) + y_{k+l}(u\ast v),
\end{align*}
for any monomials $u,v$ in $Y$.

\begin{example}\label{stuffle}
Explicitly, the homomorphism, for words not ending in $y_1$ is given by
\[ y_{n_1}y_{n_2}\cdots y_{n_k} \mapsto \zeta(n_1,\ldots, n_k).\]
For example, consider \[y_2\ast y_3 = y_2y_3 + y_3y_2 +y_5.\]
Under this map, this should correspond to $\z(2)\z(3)$, which is equal to
\begin{align*}
\sum_{0<m}\sum_{0<n}\frac{1}{m^2n^3} =&{} \sum_{0<m<n}\frac{1}{m^2n^3} + \sum_{0<n<m}\frac{1}{m^2n^3}\\
&+\sum_{0<m=n=N}\frac{1}{N^5}\\
=&{} \z(2,3)+\z(3,2)+ \z(5)
\end{align*}
as claimed.
\end{example}

In his papers \cite{hoffmanquasi1,hoffmanquasi2}, Hoffman explores defines a quasi-shuffle algebra on a countable alphabet $A$ to be the vector space $\Q\langle A\rangle$ of non-commutative polynomials equipped with a product defined recursively by
\begin{align*}
u \star 1 &= 1\star u = u,\\
au\ast bv &= a(u\ast y_lv) + b(y_k\ast v) + (a\lozenge b)(u\ast v),
\end{align*}
for monomials $u,v$, and letters $a,b\in A$, where $\lozenge$ defines a commutative, associative product on $\Q A$, the vector space spanned by elements of $A$. He goes on to show that $(\Q\langle A\rangle,\sh)\cong (\Q\langle A\rangle, \star)$, allowing us to transfer our understanding of the shuffle algebra to quasi-shuffle algebras. This can be applied to multiple zeta values, multiple $t$-values, interpolated multiple zeta values, etc, in order to prove various relations, or construct a set of stuffle-generators for $\mathcal{Z}$.

\begin{example}
It is known that, for an ordered alphabet $A$, Lyndon words form an algebra basis for $(\Q\langle A\rangle,\sh)$. Using Hoffman's isomorphism, one may show that Lyndon words also form an algebra basis for $(\Q\langle A\rangle, \star)$. In particular, we must have that $\mathcal{Z}$ is stuffle-generated by the set
\[\{\z(n_1,\ldots,n_r) \mid (n_1,\ldots,n_r)\text{ is a Lyndon word in }\mathbb{N}\text{ for the standard ordering}\}.\]
\end{example}

In a number of recent talks, Hirose and Sato propose a new shuffle-like structure describing relations among multiple zeta values. In the following, we will explain this new structure and how is may be viewed as a generalisation of Hoffman's quasi-shuffle algebras. We will define the notion of a generalised quasi-shuffle algebra and go on to reproduce most of Hoffman's initial results from \cite{hoffmanquasi1}. In particular, we show that generalised quasi-shuffle algebras are isomorphic to the standard shuffle algebra, and are hence generated by Lyndon words. 

We then apply this to the case of multiple zeta values, giving a new linear generating set, reproducing a number of known families of relations, and establishing several more, many of which can be verified using the MZV datamine \cite{datamine}.

\section{The block decomposition of multiple zeta values}
It is well known that multiple zeta values can alternatively be written as an iterated integral of $\proj$ \cite{delmixedtate} as defined by Chen \cite{chen}.

\begin{defn}
For any sequence $(a_0;a_1,\ldots,a_n;a_{n+1})\in\{0,1\}^{n+2}$, with $a_1=1, a_n=0$, define 
$$\ic(a_0;a_1,\ldots,a_n;a_{n+1}):=\int_{a_0\leq t_1\leq\cdots\leq t_n\leq a_{n+1}}\prod_{k=1}^n \frac{dt_k}{t_k-a_k}$$
\end{defn}

\begin{prop}[Chen \cite{chen}]
There exists a unique extension of $\ic(a_0;a_1,\ldots,a_n;a_{n+1})$ to all sequences such that
\begin{itemize}
\item $\ic(a_0;a_1)=1$
\item $\ic(a_0;0;a_1)=\ic(a_0;1;a_1)=0$
\item Products satisfy \begin{align*}&\ic(a_0;a_{1},\ldots,a_{m};a_{m+n+1})\ic(a_0;a_{m+1},\ldots,a_{m+n};a_N)\\
={}&\sum_{\sigma\in\Sh_{m,n}}\ic(a_0;a_{\sigma(1)},\ldots,a_{\sigma(m+n)};a_{m+n+1})\end{align*}
where $$\Sh_{m,n}:=\{\sigma\in S_{m+n}\mid \sigma(1)<\sigma(2)<\cdots<\sigma(m),\ \sigma(m+1)<\cdots<\sigma(m+n)\}$$
and $S_N$ denotes the symmetric group on $\{1,\ldots,N\}$.
\end{itemize}
\end{prop}

Via the above theorem, we may view $\ic$ as a linear function on $\qpoly$ by defining
\[\ic(0;e_{i_1}\cdots e_{i_n};1):= \ic(0;i_1,\ldots,i_n;1)\]
and 
\[\ic(e_{i_0}\cdots e_{i_{n+1}}):=\ic(i_0;i_1,\ldots,i_n;i_{n+1}).\]

Defining the mapping $\phi:(n_1,\ldots,n_r)\mapsto (0;1,\{0\}^{n_1-1},1,\ldots,1,\{0\}^{n_r-1};1)$, where $\{0\}^k$ represents the sequence of $k$ repeated zeroes, we can write any multiple zeta value as an iterated integral:
\[\z(n_1,\ldots,n_r) = (-1)^r\ic(\phi(n_1,\ldots,n_r)).\]
Using this equality, and Chen's theorem, we see that the image $\ic(\qpoly)$ is equal to the algebra of multiple zeta values.

\begin{example}\label{zeta2int}
We claim $\z(2)=-\ic(0;1,0;1)$.
\begin{align*}
\ic(0;1,0;1)=&{} \int_{0\leq t_1\leq t_2\leq 1}\frac{dt_1}{t_1-1}\frac{dt_2}{t_2}\\
=&{} -\int_{0\leq t_1\leq t_2\leq 1}dt_1\sum_{n\geq 0}t_1^n \frac{dt_2}{t_2} = -\int_{0\leq t_2\leq 1}dt_2\sum_{n\geq 0} \frac{t_2^n}{n+1}\\
={}& -\sum_{n\geq 0}\frac{1}{(n+1)^2}=-\z(2)
\end{align*}
\end{example}

In fact, using the third part of Chen's theorem, we see that $\ic:(\qpoly,\sh) \to (\mathcal{Z},\cdot)$ is an algebra morphism from the shuffle algebra to the algebra of multiple zeta values. Thus via iterated integrals, and the map from Example \ref{stuffle}, we see that relations among multiple zeta values are encoded in the combinatorics of shuffle and quasi-shuffle algebras. Here, we will introduce a third such structure.

In \cite{charthesis}, Charlton introduces the notion of the block decomposition of a word in two letters $\{x,y\}$ as follows.

\begin{defn}
A word in $\{x,y\}$ is called alternating if it is non-empty and contains no subsequences of the form $xx$ or $yy$. The block decomposition of a word $w$ is the unique minimal factorisation into alternating words. Explicitly, the block decomposition of a word $w$ is the unique factorisation $w=w_1w_2\cdots w_k$ such that each $w_i$ is alternating and the last letter of $w_i$ equals the first letter of $w_{i+1}$  for each $1\leq i <k$. 
\end{defn}

\begin{defn}
Let $Z=\{z_1,z_2,\ldots\}$ and define a $\Q$-linear isomorphism $e_0\qpoly\to \Q\langle Z\rangle$ as follows. For a word $w=e_0w_0$, let $w_1w_2\cdots w_k$ be its block decomposition. We then map $w$ to the monomial $z_{|w_1|}z_{|w_2|}\ldots z_{|w_k|}$, where $|u|$ denotes the length of the word $u$.
\end{defn}

Using this isomorphism, we can view $\ic$ as a linear map $\Q\langle Z\rangle \to \mathcal{Z}$. Note that, for reasons of parity, if $i_1+\cdots+ i_k \equiv k+1$ (mod $2$), then $\ic(z_{i_1}\ldots z_{i_k})=0$, as the corresponding sequence of $0$s and $1$s begins and ends on $0$.

We define an algebra structure on $\Q\langle Z\rangle$ via the following recursive formulae.
\begin{align*}
w\bsh 1 &= 1\bsh w = w\\
z_au\bsh z_bv &:= z_a(u\bsh z_bv) + z_b(z_au\bsh v) - \zeta_{a+ b}(u\bsh v)
\end{align*}
for $a,b\in \mathbb{N}$, and $u,v,w$ monomials in $k\langle Z\rangle$. Here $\zeta_{a}$ is the linear map defined by
\begin{align*}
\zeta_a(z_bw) &:= z_{a+b}w,\\
\zeta_a(1) &:= 0.
\end{align*}
This is a special case of Definition \ref{gqshuff}, and defines an associative, commutative product on $\Q\langle Z\rangle$ called the block shuffle product. We then have the following result, initially proposed by Hirose and Sato in a series of talks, and established in a recent preprint \cite{hirosesatoblock}.

\begin{prop}\label{fullblockshuff}
Up to some choice of normalisation, $\ic(\Q\langle Z\rangle^{\bsh2})=0$. That is to say, for all monomials $u,v$
\[\ic(u\bsh v)=0\]
where $\ic$ is viewed as a linear function on $\Q\langle Z\rangle$.
\end{prop}

\begin{example}
By considering 
\[z_4\bsh z_3z_2z_2= z_4z_3z_2z_2 + z_3z_4z_2z_2 + z_3z_2z_4z_2 + z_3z_2z_2z_4 - z_9z_2 -z_3z_8,\]
 we obtain the following (genuine) relation among multiple zeta values:
\[-\z(2,1,2,1,3)+\z(3,2,1,3)+\z(3,1,2,3)+\z(3,1,3,2)=\z(2,2,2,3)+\z(3,2,2,2).\]
\end{example}

In this paper we show that this result allows us to easily reproduce several families of known relations among multiple zeta values, as well as provide a number of generalisations.

\begin{example}
The following families of results - due to Broadhurst, Bradley, Borwein, and Lisonek \cite{bbbl}, and Bowman and Bradley \cite{bowmanbradley02} respectively - follow from Conjecture \ref{fullblockshuff} and the discussions below.
\begin{align*}
\z(\{1,3\}^k)&=\frac{\z(\{2\}^{2k})}{2k+1}=\frac{2\pi^{4k}}{(4k+2)!},\\
\z(\{2\}^k\sh\{1,3\}^n)&=\frac{\pi^{4n+2k}}{(2n+1)(4n+2k+1)!}\binom{2n+k}{k}.
\end{align*}
\end{example}

\section{Generalised  quasi-shuffle algebras}
Inspired by Hirose and Sato's definition of the block shuffle product, we define a generalisated quasi-shuffle algebra as follows. Let $Z$ be a countable set of letters, and $k$ a field of. Let $-\lozenge -: kZ\otimes kZ\to kZ$ be a $k$-linear, commutative, associative product on the vector space $kZ$. We furthermore define for each $a\in Z$ the linear map
\begin{align*}
L_a: kZ &\to kZ,\\
 z &\mapsto a\lozenge z.
\end{align*}
We extend each $L_a$ to a map
\[ k\langle Z \rangle\to k\langle Z\rangle\]
by defining 
\begin{align*}
L_a(xw) &:= L_a(x)w, \ x\in Z,\\
L_a(1) &:= 0.
\end{align*}

\begin{defn}\label{gqshuff}
We define a generalised quasi-shuffle product $\star$ on $k\langle Z\rangle$ to be a bilinear product given recursively by
\begin{align*}
w\star 1 &= 1\star w = w\\
au\star bv &:= a(u\star bv) + b(au\star v) - L_{a\lozenge b}(u\star v)
\end{align*}
for $a,b\in Z$, and $u,v,w$ monomials in $k\langle Z\rangle$.
\end{defn}

In the case $Z=\{z_1,z_2,\ldots\}$ and $z_m\lozenge z_n:=z_{m+n}$, we obtain precisely the block shuffle algebra of Hirose and Sato. In this case, we denote the product by $\bsh$ rather than $\star$.

The reader can readily verify that $(k\langle Z\rangle, \star)$ is a commutative associative algebra, e.g. by induction on the length of words. While Hoffman's results cannot be immediately applied, the same methods apply. We assume that $k=\Q$, though the results may be extended easily to any field of characteristic zero.

\begin{defn}
A composition $I$ of $n$ is a sequence of positive integers $(i_1,\ldots,i_l)$ such that $i_1+\cdots +i_l=n$. Given a composition $I$ of $n$ into $l$ parts and a composition $J$ of $l$ into $k$ parts, we define the product composition:
$$J\circ I := (i_1+\cdots+ i_{j_1},i_{j_1+1}+\cdots+i_{j_1+j_2},\ldots,i_{j_1+\cdots+j_{k-1}+1}+\cdots+i_{j_1+\cdots+j_k}).$$
Denote by $\mathcal{C}(n)$ the set of compositions of n.
\end{defn}

We define an action of compositions on $\Q\langle Z\rangle$ as follows. Define $[z_{a_1}\ldots z_{a_k}]:=z_{a_1}\lozenge\cdots\lozenge z_{a_k}$, and given a composition $I$ of $n$, define
$$I[z_{a_1}\ldots z_{a_n}]:=[z_{a_1}\ldots z_{a_{i_1}}][z_{a_{i_1}+1}\ldots z_{a_{i_1+i_2}}]\cdots [z_{a_{i_1+\cdots+i_{l-1}+1}}\ldots z_{a_n}]$$
and $I[w]=0$ for any words not of length $n$.

\begin{prop}\label{Tanh}
Let $\Psi_{\tanh}:k\langle Z\rangle\to\Q\langle Z\rangle$ be the linear map with $\Psi_{\tanh}(1)=1$ and, for $w$ a word of length $n$
$$\Psi_{\tanh}(w) = \sum_{(i_1,\ldots,i_l)\in\mathcal{C}(n)}c_{i_1}\ldots c_{i_l}(i_1,\ldots,i_l)[w]$$
where $c_j$ is the coefficient of $x^j$ in the Taylor expansion of $\tanh(x)$. Then $\Psi_{\tanh}$ is an algebra isomorphism
$$\Psi_{\tanh}:(\Q\langle Z\rangle,\sh)\to(\Q\langle Z\rangle,\star).$$
\end{prop}
To prove this, we require the following two results. The first is due to Hoffman \cite{hoffmanquasi1}.

\begin{lem}
Let $f(z)=c_1z+c_2z^2+\cdots$ be a function analytic at $0$, with $c_1\neq 0$, and $c_i\in\Q$ for all $i$. Let $f^{-1}(z)=b_1z+b_2z^2+\cdots$ be its inverse. Then the map $\Psi_f:\Q\langle Z\rangle\to\Q\langle Z\rangle$ given by 
$$\Psi_f(w)=\sum_{(i_1,\ldots,i_l)\in\mathcal{C}(n)}c_{i_1}\ldots c_{i_l}(i_1,\ldots,i_l)[w]$$
for words of length $n$, and extended linearly to $\Q\langle Z\rangle$, has inverse $\Psi_f^{-1}=\Psi_{f^{-1}}$.
\end{lem}

This lemma allows us to establish $\Psi_{\tanh}$ as the inverse map of a map $\Psi_{\tanh^{-1}}$. Rather than directly show $\Psi_{\tanh}$ to be a homomorphism, we instead show that its inverse $\Psi_{\tanh^{-1}}$ is.

\begin{prop}\label{tanhmorphism}
The map $\Psi_{\tanh^{-1}}$ is an algebra homomorphism $(\Q\langle Z\rangle, \star)\to(\Q\langle Z\rangle, \sh)$.
\end{prop}

Since the proof of this proposition is somewhat tedious, we delay it until the end of this article. As a corollary to Proposition \ref{Tanh}, we obtain the following.

\begin{cor}\label{blockLyndon}
$(\Q\langle Z\rangle,\star)$ is the free polynomial algebra on the Lyndon words.
\end{cor}
\begin{proof}
Hoffman's proof of Theorem 2.6 \cite{hoffmanquasi1} applies exactly. We sketch the proof here, but refer the reader to Hoffman's proof for further detail. The proof proceeds by induction on the length of a word. Suppose $w$ is a word of length $l$. As $(\Q\langle Z\rangle, \sh)$ is a free polynomial algebra on Lyndon words, there exist Lyndon words $w_1,\ldots,w_n$ and a polynomial $P$ such that 
$$w=P(\Psi_{\tanh}(w_1),\Psi_{\tanh}(w_2),\ldots,\Psi_{\tanh}(w_n))$$
where $P$ is considered as a $\bsh$-polynomial. Since the shuffle product preserves length, and $\Psi_{\tanh^{-1}}(w)$ has terms with length at most $l$, we can assume every term of $P(w_1,\ldots,w_n)$ has length at most $l$, where $P$ is considered as a $\sh$-polynomial. But then
$$w-P(w_1,\ldots,w_n)=P(\Psi_{\tanh}(w_1),\ldots,\Psi_{\tanh}(w_n))-P(w_1,\ldots,w_n)$$
has only terms of length less than $l$, and so can be written as a $\star$-polynomial of Lyndon words.
\end{proof}

\begin{cor}
Let $Z_n$ be the $\Q$-span of words of length $n$ in $\Q\langle Z\rangle$, and let
\[Z(x)=\sum_{n\geq 0}(\dim Z_n)x^n\]
be the Poincar{\'e} series. Define $c_n$ by
\[ x\frac{d}{dx}\log Z(x) = \sum_{n \geq 1} c_n x^n.\]
Then the number of Lyndon words of length $n$ in $(\Q\langle Z\rangle,\star)$ is given by
\[L_n=\frac{1}{n}\sum_{d\mid n}\mu\left(\frac{n}{d}\right)c_d.\]
\end{cor}
\begin{proof}
As in Proposition 2.7 of \cite{hoffmanquasi1}, we must have $Z(x)=\prod_{n\geq 1}(1-x^n)^{-L_n}$, following Corollary \ref{blockLyndon}. Taking logarithms, differentiating, and applying M{\"o}bius inversion gives the result.
\end{proof}

\section{Applications to multiple zeta values}

Following the work of Hoffman and Ihara, we prove some additional properties of $\Psi_{\tanh}$, which will find application to multiple zeta values in the following section. 

We first recall one of their results, specialised to the case of block shuffle. In all that follows, $\lambda$ is a formal parameter, and we extend $\Psi_f$ by $\Psi_f(\lambda)=\lambda$.

\begin{defn}
Define $\lozenge:\Q Z\otimes \Q Z\to\Q Z$ by $z_m\lozenge z_n:=z_{m+n}$. Then, for any $f(z)=c_1z+c_2z^2+\cdots$, define
$$f_\bullet(\lambda w):= \sum_{i=1}^\infty \lambda^ic_iw^{\bullet i}$$
for $\bullet\in\{\sh,\bsh\}$ and $w\in\Q\langle Z\rangle$, or $\bullet=\lozenge$ and $w\in\Q Z$.
\end{defn}

\begin{remark}
In a slight abuse of notation, we shall write $\exp_\bullet(w)$ for $1+f_\bullet(w)$ where $f(z)=e^z-1$; and $\log_\bullet(1+w)$ for $f_\bullet(w)$, where $f(z)=\log(1+z)$, and similarly for $\tanh^{-1}_\bullet(1+w)$. Note that 
$$log_\bullet(\exp_\bullet(\lambda w))=\lambda w\text{ and } \exp_\bullet(\log_\bullet(1+\lambda w))=1+\lambda w.$$
\end{remark}

\begin{prop}[Theorem 5.1 \cite{hoffmanquasi2}]\label{hoffmangeometric}
For any $f(z)=c_1z+c_2z^2+\cdots$ and $z\in\Q Z[[\lambda]]$,
$$\Psi_f\left(\frac{1}{1-\lambda z}\right) = \frac{1}{1-f_\lozenge(\lambda z)}.$$
\end{prop}

We also need a modification of a lemma due to Hoffman and Ihara.

\begin{lem}\label{expTanh}
For $z\in\Q Z[[\lambda]]$
$$\exp_{\bsh}(\lambda z) = \Psi_{\tanh}\left(\frac{1}{1-\lambda z}\right).$$
\end{lem}
\begin{proof}
Since $\Psi_{\tanh}:(\Q\langle Z\rangle, \sh)\to(\Q\langle Z\rangle,\bsh)$ is an algebra isomorphism, we must have that $\Psi_{\tanh}\circ f_{\sh} = f_{\bsh}\circ\Psi_{\tanh}$. Thus, as $\Psi_{\tanh}|_{\Q Z}=id$, we have
$$exp_{\bsh}(\lambda z)=exp_{\bsh}(\Psi_{\tanh}(\lambda z))=\Psi_{\tanh}(exp_{\sh}(\lambda z))=\Psi_{\tanh}\left(\frac{1}{1-\lambda z}\right)$$
where we have used that 
$$exp_{\sh}(\lambda z)=\sum_{n=0}^\infty\lambda^n\frac{z^{\sh n}}{n!}=\sum_{n=0}^\infty\lambda^n\frac{n!z^{n}}{n!}=\sum_{n=0}^\infty\lambda^nz^n.$$
\end{proof}
Thus we can show the following

\begin{prop}\label{expgeo}
For $z\in\Q A[[\lambda]]$
$$exp_{\bsh}(\tanh^{-1}_\lozenge(1+\lambda z))=\frac{1}{1-\lambda z}.$$
\end{prop}
\begin{proof}
By Lemma \ref{expTanh}, this is equivalent to showing that
$$\Psi_{\tanh}\left(\frac{1}{1-\tanh^{-1}_\lozenge(1+\lambda z)}\right)=\frac{1}{1-\lambda z}.$$
However, this follows immediately from the statement of Proposition \ref{hoffmangeometric} for $f=\tanh^{-1}$.
\end{proof}

\begin{cor}
Recall we have a surjective linear map $\ic:\Q\langle Z\rangle \to \mathcal{Z}$, mapping a word to its corresponding iterated integral. Denote by $L(Z)$ the $\Q$-span of the set of Lyndon words in $Z$. Then, assuming Conjecture \ref{fullblockshuff}, $\mathcal{Z}=\ic(L(Z))$.
\end{cor}
\begin{proof}
Every word in $\Q\langle Z\rangle$ can be written as a $\bsh$-polynomial in the Lyndon words. Conjecture \ref{fullblockshuff} tells us that the image of any terms of degree greater than $1$ in this polynomial is $0$, and hence $\ic(w)$ is the image of the linear part, i.e. $\ic(w)\in\ic(L(Z))$.
\end{proof}

\begin{example}
This suggests that, in weight $5$, multiple zeta values are spanned by the (appropriately regularised) set
\[\{\z(2,2,1),\z(2,1,2),\z(3,2),\z(1,3,1),\z(3,1,1),\z(4,1),\z(5),\z(2,1,1,1),\z(1,1,1,1,1)\}.\]
Up to application of duality, this contains the Hoffman elements $\{\z(2,1,2)=\z(2,3),\z(3,2)\}$, which are known to be a linear spanning set by the work of Brown \cite{brownmixedtate}. Indeed, in general $\ic(L(Z))$ contains the block degree one part of Hoffman spanning set (up to duality), but in higher block degree gives alternative generating elements. Note also that this does not give us a basis, nor is it currently known how to extract a basis from this collection.
\end{example}

\begin{cor}\label{quasipower}
For any $z\in\Q Z$ and any $n>1$,
$$\ic(z^{2k+1})=\frac{1}{2k+1}\ic(z^{\lozenge 2k+1})\in\ic\left(\Q Z\right) \subset\Q[\pi^2]$$
and 
$$\ic(z^{2k})=0$$
for all $z\in\Q Z$.
\end{cor}
\begin{proof}
Taking the image of the equality in Proposition \ref{expgeo}, we obtain
$$\ic\left(1+ \tanh^{-1}_\lozenge(1+\lambda z)+\sum_{n\geq 2}\frac{\tanh^{-1}_\lozenge(\lambda z)^{\bsh n}}{n!}\right)=\sum_{n\geq 0}\lambda^n\ic(z^n).$$
As $\ic$ kills $\bsh$-products, the left hand side is just 
$$\sum_{k\geq 0}\frac{\lambda^{2k+1} z^{\lozenge 2k+1}}{2k+1}$$
The result follows upon comparing coefficients of $\lambda^n$.
\end{proof}

\begin{remark}
As $\ic(w)=0$ if the length of $w$ and the weight of $w$ are of the same parity, we see that the projection of $z$ onto $\bigoplus_{i=1}^\infty \Q z_{2i}$ must be non-zero for the statement to be non-trivial.
\end{remark}

Assuming Conjecture \ref{fullblockshuff}, we can produce the following examples of explicit identities, many of which are known to be true.

\begin{ex}
Consider $z=z_2$, then $\ic(z_2^{2k+1})=\frac{1}{2k+1}\ic(z_{4k+2})$. This is a well known result from the work of Borwen, Bradley, Broadhurst, and Lisonek \cite{bbbl}:
$$\z(\{1,3\}^k)=\frac{\z(\{2\}^{2k})}{2k+1}=\frac{2\pi^{4k}}{(4k+2)!},$$ 
where $\{n_1,\ldots,n_r\}^k$ is the sequence consisting of $n_1,\ldots,n_r$ repeated $k$ times. Similarly, Theorem 2 of \cite{bbbl} tells us that $\z(2\sh\{1,3\}^n)=\frac{\pi^{4n+2}}{(4n+3)!}$ and taking $z=z_2+z_4$, and considering the weight $4n+2$ part of $z^{2n+1}$, we obtain precisely this. 

More generally, by considering the weight $4n+2k$ part of $z^{2n+1}$ for $z=z_2+z_4+\cdots+ z_{2k+2}$, we obtain another result due to Bowman and Bradley \cite{bowmanbradley02}.

\begin{thm}\label{213shuff}
For all non-negative $n,k$.
\[\z(\{2\}^k\sh\{1,3\}^n)=\frac{\pi^{4n+2k}}{(2n+1)(4n+2k+1)!}\binom{2n+k}{k}\]
\end{thm}

By considering other weights and taking $z$ to be some other linear combination, we obtain that sums over certain subsets of the set of shuffles also evaluate to elements of $\Q[\pi^2]$.
\end{ex}

\begin{prop}\label{bunchsof2}
Let $S_{n,k,p}$ denote the set of words in $\{1,2,3\}$ appearing in the shuffle product $\{2\}^k\sh\{1,3\}^n$ containing at least one group of $p$ adjacent $2$s, and no group of $p+1$ adjacent $2$s. Then
$$\sum_{u\in S_{n,k,p}}\z(u)\in\Q\pi^{2k+4n}.$$
\end{prop}
\begin{proof}
Corollary \ref{quasipower} tells us that, for any $n,p\geq 1$ 
$$\ic((z_2+\cdots + z_{2p+2})^{2n+1})=\frac{1}{2n+1}\ic((z_2+\cdots+z_{2p+2})^{\lozenge 2n+1})$$
As $(z_2+\cdots+z_{2p+2})^{\lozenge 2n+1}\in\bigoplus_{i\geq 1}\Q z_{2i}$, and $\ic(z_{2i})=\z(\{2\}^{i-1})$, we must have that $\ic((z_2+\cdots + z_{2p+2})^{2n+1})\in\Q[\pi^2]$. To be precise
$$\ic((z_2+\cdots + z_{2p+2})^{2n+1})=\frac{1}{2n+1}\sum_{i\geq 0} |\mathcal{P}_{i+1,2n+1,p+1}| \z(\{2\}^{i})$$
where $\mathcal{P}_{k,N,r}$ denotes the set of compositions of $k$ into $N$ parts of size at most $r$. Hence
$$\ic((z_2+\cdots + z_{2p+2})^{2n+1})-\ic((z_2+\cdots+z_{2p})^{2n+1})\in\Q[\pi^2].$$
Letting $T_{n,k,p}$ be the set of all monomials of block degree $2n$, weight $2k+4n$, containing at least one $z_{2p+2}$, this is precisely the statement that
$$\sum_{w\in T_{n,k,p}}\ic(w)\in\Q\pi^{2k+4n}.$$
Translating this into the language of multiple zeta values, we note that every $z_{2k+2}$ corresponds to a group of exactly $k$ adjacent $2$s, and so elements of $T_{n,k,p}$ correspond exactly to elements of $S_{n,k,p}$. Thus, letting $\mathcal{P}_{k,N,r}^+$ denote the set of compositions of $k$ into exactly $N$ parts of size at most $r$ and containing at least one part of size $r$, we have
$$\sum_{u\in S_{n,k,p}}\z(u)=\frac{1}{2n+1} |\mathcal{P}_{k+2n+1,2n+1,p+1}^+|\z(\{2\}^{k})\in\Q\pi^{2k+2n}.$$
\end{proof}

\begin{ex} For example, if we take $z=z_4$, we obtain that 
$$\z(\{2,1,2,3\}^k,2)=\frac{\pi^{8k+2}}{(2k+1)(8k+3)!}=\frac{4\pi^{8k+2}}{(8k+4)!}.$$ 
This corresponds to $S_{k,k+1,1}$ and agrees with predictions made using the datamine \cite{datamine}.
\end{ex}

 \begin{remark}
We can actually refine this result significantly: by considering $z:=a_2z_2+\cdots+a_{2p+2}z_{2p+2}$, and allowing the $a_{2i}$ to vary freely, we see that we must have 
$$\sum_{u\in I_{i_1,\ldots,i_{p+1},n,w}}\ic(u)\in\Q\pi^w$$
where $I_{i_1,\ldots,i_{p+1},n,w}$ is the set of words of degree $2n+1$ and weight $w$ with $\\deg_{z_{2j}}(u)=i_j$ for $j\leq p+1$ and $\deg_{z_k}(u)=0$ for all other $k$. This implies
$$\sum_{(n_1,\ldots,n_k)\in J_{i_1,\ldots,i_{p+1},n,w}}\z(n_1,\ldots,n_k)\in\Q\pi^w$$
where $J_{i_1,\ldots,i_{p+1},n,w}$ is the set of tuples $(n_1,\ldots,n_k)$ with $n_i\in\{1,2,3\}$, satisfying the following:
\begin{enumerate}
    \item $n_1+\cdots+n_k=w$,
    \item Among the integers $n_1,\ldots,n_k$, exactly $n$ are 1 and exactly $n$ are 3,
    \item Omitting 2s, the sequence $(n_1,\ldots,n_k)$ becomes $(\{1,3\}^n)$,
    \item The sequence $(n_1,\ldots,n_k)$ contains exactly $i_j$ groups of exactly $j-1$ adjacent 2s, for $j>1$.
\end{enumerate}
For example, for $(i_1,i_2,n,w)=(2,1,3,6)$, we find that
\begin{align*}
    \ic(z_2z_2z_4)+\ic(z_2z_4z_2)+\ic(z_4z_2z_2) &= \z(1,3,2)+\z(1,2,3)+\z(2,1,3)\\
    &=\z(2,2,2)=\frac{\pi^6}{5040}.
\end{align*}
\end{remark}

\begin{remark}
Using the theory of mixed Tate motives, multiple zeta values may be lifted to motivic analogues, called motivic multiple zeta values \cite{brownmixedtate}. These may thought of a formal algebraic analogues of multiple zeta values, satisfying only relations coming from geometry. In particular, motivic multiple zeta values satisfy the standard transcendence conjectures for multiple zeta values such as being graded by weight, or algebraic independence of odd single zeta values. As Proposition \ref{fullblockshuff}  in fact holds for motivic multiple zeta values, then all of the above results also lift to motivic relations in keeping with the author's results showing that block shuffle relation holds motivically modulo terms of lower $z$-degree \cite{keilthyblock1}.
\end{remark}

\section{Proof of Proposition \ref{tanhmorphism}}

To the best of the author's knowledge, the easiest way to establish this is to show that the dual of $\Psi_{\tanh^{-1}}$ is a coalgebra homomorphism
\[\Phi: (\Q\langle Z\rangle, \Delta) \to (\Q\langle Z\rangle, \bDelta)\]
where the coproducts are given by the duals  of $\sh$ and $\star$ respectively:
\[\Delta(z) := z\otimes 1 + 1\otimes z \text{ for all }z\in Z\]
and
\[\bDelta(z) :=\sum_{k\geq 0} (-1)^k \sum_{\substack{z_{i_1},\ldots,z_{i_{2k+1}}\in Z\\ z_{i_1}\lozenge\cdots\lozenge z_{i_{2k+1}}=z}} z_{i_1}\ldots z_{i_k}\otimes z_{i_{k+1}}\ldots z_{i_{2k+1}} + z_{i_{k+1}}\ldots z_{i_{2k+1}}\otimes z_{i_1}\ldots z_{i_k}\]
for all $z\in Z$. We also have $\Delta(1)=\bDelta(1)=1\otimes 1$.
In particular, it suffices to show that $\Phi(z)$  is primitive with respect to $\bDelta$. Let us first compute $\Phi(z)$. Consider $\Q\langle Z\rangle$ as its own graded dual via the pairing
\[\langle u,v\rangle =\delta_{u,v}\]
for monomials $u,v$. Then
\begin{align*}
\Phi(z) &= \sum_{v}\langle \Phi(z),v\rangle v\\
&=\sum_v\langle z,\Psi_{\tanh^{-1}}(v)\rangle v
\end{align*}
Letting $v=z_{a_1}\ldots z_{a_n}$, we have that
\[\Psi_{\tanh^{-1}}(v)=\sum_{\substack{(i_1,\ldots,i_l)\in\mathcal{C}(n)\\ i_j \text{ odd}}}\frac{1}{i_1\ldots i_l}(i_1,\ldots,i_l)[v].\]
In particular, the only terms of this sum contained in $\Q Z$ are those for which $l=1$ and $i_1=n$, and so
\[\Phi(z) =\sum_{k\geq 0}\frac{1}{2k+1}\sum_{\substack{ z_{a_1},\ldots, z_{a_{2k+1}}\in Z\\ z_{a_1}\lozenge\cdots\lozenge z_{a_{2k+1}}=z}}z_{a_1}\ldots z_{a_{2k+1}}.\]

We claim this is primitive with respect to $\bDelta$. We first write $\bDelta(\Phi(z))=\sum_{r,s\geq 0}w_{r,s}$, where
\[w_{r,s}\in\text{Span}\{z_{i_1}\ldots z_{i_r}\otimes z_{j_1}\ldots z_{j_s}\}\]
is the degree $(r,s)$ component. As $\bDelta$ is cocommutative, it suffices to show that $w_{r,s}=0$ for all $0<r\leq s$.

Let us consider the contribution of $\bDelta(z_{i_1}\ldots z_{i_{2k+1}})$ to $w_{r,s}$. Every term in this contribution can be identified with a triplet $(\textbf{j},C,d)$ consisting of a sequence of integers
\[ 0<j_1<\cdots <j_m\leq 2k+1,\]
a composition  $C=(c_1,\ldots,c_m)$ of $r$ into $m$ parts, and a sequence $d=\{d_i\}\in\{\pm 1\}^m$ such that
\[2k+1-m+r+\sum_{i=1}^m d_i = s.\]

The sequence $\textbf{j}$ determines which $z_{i_1},\ldots,z_{i_{2k+1}}$ contribute terms to the left hand side of the tensor products in $w_{r,s}$. $C$ determines the degree contributed, and $d$ determines the degree contributed to the right hand side. More precisely, $(\textbf{j},C,d)$ determines the product $Z_{(\textbf{j},C,d)}^{z_{i_1}\ldots z_{i_{2k+1}}}$ given by
\[\left(\prod_{u=1}^{j_1-1}(1\otimes z_{i_u})\right)\left(\prod_{t=1}^m\left(\sum_{z_{p_1}\lozenge\cdots\lozenge z_{p_{2c_t+d_t}}=z_{i_{j_t}}}z_{p_1}\ldots z_{p_{c_t}}\otimes z_{p_{c_t+1}}\ldots z_{p_{2c_t+d_t}}\prod_{v=j_{t}+1}^{j_{t+1}-1}(1\otimes z_{i_v})\right)\right).\]
Here we take $j_{m+1}=2k+2$. The sign with which this term appears is uniquely determined by $(C,d)$ to be 
\[\prod_{i=1}^n d_i(-1)^{c_i} = (-1)^r\prod_{i=1}^m d_i.\]

Letting $\mathcal{C}(r,m)$ be the set of compositions of $r$ into exactly $m$ parts, we can thus write $w_{r,s}$ as the sum:
\[w_{r,s}=\sum_{k\geq 1}\frac{1}{2k+1}\sum_{m=1}^r\sum_{\substack{z_{i_1}\lozenge\cdots\lozenge z_{i_{2k+1}}=z\\ 0<j_1<\cdots <j_m\leq 2k+1}}\sum_{C\in\mathcal{C}(r,m)} \sideset{}{'}\sum_{d\in\{\pm 1\}^m} (-1)^r\left(\prod_{t=1}^m d_t \right)Z_{(\textbf{j},C,d)}^{z_{i_1}\ldots z_{i_{2k+1}}}\]
where the final sum is restricted to those $d$ such that $2k+1-m+r+\sum_{i=1}^m=s$. We can rewrite this sum as
$$w_{r,s}=\sum_{m=1}^r\sum_{C\in\mathcal{C}(r,m)}\sum_{d\in\{\pm 1\}^m} (-1)^r\left(\prod_{t=1}^m d_t\right) \frac{1}{s+m-r-\sum_{t=1}^md_t}\sum_{\substack{z_{i_1}\lozenge\cdots \lozenge z_{i_{2k+1}}=z\\ 0<j_1<\cdots <j_m\leq 2k+1}}Z_{(\textbf{j},C,d)}^{z_{i_1}\ldots z_{i_{2k+1}}}.$$

Performing the sums over the $i$ and $j$ indices, we obtain
\begin{equation*}
\begin{split}
w_{r,s}= &\sum_{m=1}^r\sum_{C\in\mathcal{C}(r,m)}\sum_{d\in\{\pm 1\}^m} (-1)^r\left(\prod_{t=1}^m d_t\right) \frac{1}{s+m-r-\sum_{t=1}^md_t}\binom{s+m-r-\sum_{t=1}^md_t}{m}\\
&\times \sum_{z_{a_1}\lozenge\cdots\lozenge z_{a_{r+s}}=z}z_{a_1}\ldots z_{a_r}\otimes z_{a_{r+1}}\ldots z_{a_{r+s}}.
\end{split}
\end{equation*}

Hence, it suffices to compute
$$\sum_{m=1}^r\sum_{C\in\mathcal{C}(r,m)}\sum_{d\in\{\pm 1\}^m}\left(\prod_{t=1}^md_t\right)\frac{1}{s+m-r-\sum_{t=1}^md_t}\binom{s+m-r-\sum_{i=1}^md_t}{m}.$$
Note that if $\sum_{t=1}^md_t=m-q$, then $\prod_{t=1}^m d_t=(-1)^q$, and so we can replace the sum over $d\in\{\pm 1\}^m$ with a sum over $q$, and perform the sum over compositions to obtain that this sum is equal to
$$\sum_{m=1}^r\sum_{q=0}^m (-1)^{r+q}\frac{1}{s-r+q}\binom{s-r+q}{m}\binom{r+m-1}{m-1}\binom{m}{q}.$$
We will evaluate the sum
$$Q_m:=\sum_{q=0}^m (-1)^q\frac{1}{s-r+q}\binom{s-r+q}{m}\binom{m}{q}.$$
Denote by $[x^i]f(x)$ the coefficient of $x^i$ in $f(x)$, where $f$ is a polynomial in $x$. Then we have
\begin{equation*}
\begin{split}
Q_m&=[x^m]\sum_{q=0}^m\frac{(-1)^q}{s-r+q}(x+1)^{s-r+q}\binom{m}{q}\\
&=[x^m]\int_{-1}^x(y+1)^{s-r-1} \sum_{q=0}^m (-1)^q(y+1)^q\binom{m}{q}dy\\
&=[x^m]\int_{-1}^x (y+1)^{s-r-1}(-y)^m dy\\
&=0\text{, as the term of minimal degree is $x^{m+1}$}.
\end{split}
\end{equation*}
Hence, $w_{r,s}=0$ for all $0<r\leq s$, and thus $\Phi(z)$ is primitive. Hence $\Phi$ is a coalgebra homomorphism, and so $\Psi_{\tanh^{-1}}$ is an algebra homomorphism.
\hfill\qed

\section{Hopf algebra structure}
In interest of completeness, we will also consider Hoffman's Hopf algebraic results. In particular, we claim that $(\Q\langle Z\rangle,\star)$ has the structure of a Hopf algebra when equipped with the deconcatenation coproduct 
\begin{align*}
\decon 1 &:= 1\otimes 1\\
\decon z_{i_1}z_{i_2}\ldots z_{i_r} &:= \sum_{k=0}^r z_{i_1}\ldots z_{i_k}\otimes z_{i_{k+1}}\ldots z_{i_r},
\end{align*}
counit
\begin{align*}
\epsilon(1) &:=1\\
\epsilon(w) &:= 0 \text{ for all }w\text{ a non empty word},
\end{align*}
and antipode
\begin{equation*}
S(z_{i_1}\ldots z_{i_r}) := (-1)^rz_{i_r}\ldots z_{i_1}.
\end{equation*}
We will also show that $\Psi_{\tanh}:(\Q\langle Z\rangle,\sh,\decon)\to (\Q\langle Z\rangle, \star,\decon)$ is a Hopf algebra isomorphism.

\begin{thm}
The algebra $(\Q\langle Z\rangle,\star)$ equipped with the above coproduct, counit, and antipode is a Hopf algebra.
\end{thm}
\begin{proof}
It is clear that $(\Q\langle Z\rangle, \decon,\epsilon)$ is a coalgebra, so it remains to show that $\decon$ and $\epsilon$ are $\star$-homomorphisms, and that $S$ satisfies
\[\sum_{k=0}^r S(z_{i_1}\ldots z_{i_k})\star z_{i_{k+1}}\ldots z_{i_r} = \sum_{k=0}^r z_{i_1}\ldots z_{i_k}\star S(z_{i_{k+1}}\ldots z_{i_r})=0.\]

By considerations of length, it is clear that $\epsilon$ is a $\star$-homomorphism. Since 
\[\decon(1\star z_{i_1}\ldots z_{i_r}) = \decon(z_{i_1}\ldots z_{i_r}\star 1) = \decon(1)(\star\otimes \star)\decon(z_{i_1}\ldots z_{i_r})\]
we may induct on the length. Suppose $\decon(u\star v) = \decon(u)(\star\otimes\star)\decon(v)$ for all pairs of monomials $(u,v)$ of combined length less than $n$. Consider a pair of monomials $(z_au,z_bv)$ of combined length $n$.
Then note that
\begin{align*}
\decon(z_aw) &= 1\otimes z_aw + (z_a\otimes 1)\decon(w),\\
\decon(L_aw) &= 1\otimes L_aw + (L_a\otimes 1)\decon(w).
\end{align*}
As such, the recursive formula for $\star$ gives
\begin{align*}
\decon(z_au\star z_bv) ={}&{}\decon(z_a(u\star z_bv))+\decon(z_b(z_au\star v))-\decon(L_{a+b}(u\star v))\\
={}&{} 1\otimes (z_a(u\star z_bv)+ z_b(z_au\star b) - L_{a+b}(u\star v))\\
&+ (z_a\otimes 1)\decon(u\star z_bv) + (z_b\otimes 1)\decon(z_au\star v)\\
&- (L_{z_a\lozenge z_b}\otimes 1)\decon(u\star v)\\
={}&{} 1\otimes (z_au\star z_bv) + (z_a\otimes 1)\left(\decon(u)(\star\otimes \star)\decon(z_bv)\right)\\
&+ (z_b\otimes 1)\left(\decon(z_au)\bstar\decon(v)\right)\\
&-(L_{z_a\lozenge z_b}\otimes 1)\left(\decon(u)\bstar\decon(v)\right),
\end{align*}
where by $\bstar$ we denote the bilinear operator given by
\[(u_1\otimes v_1)\bstar {}(u_2\otimes v_2):= (u_1\star u_2)\otimes (v_1\otimes v_2).\]
Note that 
\begin{align*}
(z_{a}u_1\otimes u_2)\bstar {}(z_{b}v_1\otimes v_2)={}&{} (z_{a}\otimes 1)\left((u_1\star z_{b}v_1)\otimes (u_2\star v_2)\right)\\
&+ (z_b\otimes 1)\left((z_au_1\star v_1)\otimes (u_2\star v_2)\right)\\
&- (L_{z_a\lozenge z_b}\otimes 1)\left((u_1\star v_1)\otimes(u_2\star v_2)\right)
\end{align*}
and hence
\begin{align*}
&\sum_{k=0}^r\sum_{l=0}^q (z_{i_1}\ldots z_{i_k}\otimes z_{i_{k+1}}\ldots z_{i_r})\bstar {}(z_{j_1}\ldots z_{j_l}\otimes z_{j_{l+1}}\ldots z_{j_q})\\
={}&{} 1\otimes (z_{i_1}\ldots z_{i_r}\star z_{j_1}\ldots z_{j_q})\\
&+ (z_{i_1}\otimes 1)\sum_{k=1}^r\sum_{l=0}^q(z_{i_2}\ldots z_{i_k}\otimes z_{i_{k+1}}\ldots z_{i_r})\bstar {}(z_{j_1}\ldots z_{j_l}\otimes z_{j_{l+1}}\ldots z_{j_q})\\
&+ (z_{j_1}\otimes 1)\sum_{k=0}^r\sum_{l=1}^q(z_{i_1}\ldots z_{i_k}\otimes z_{i_{k+1}}\ldots z_{i_r})\bstar {}(z_{j_2}\ldots z_{j_l}\otimes z_{j_{l+1}}\ldots z_{j_q})\\
&-(L_{z_{i_1}\lozenge z_{j_1}}\otimes 1)\sum_{k=1}^r\sum_{l=1}^q(z_{i_2}\ldots z_{i_k}\otimes z_{i_{k+1}}\ldots z_{i_r})\bstar {}(z_{j_2}\ldots z_{j_l}\otimes z_{j_{l+1}}\ldots z_{j_q}).
\end{align*}
In particular
\begin{align*}
\decon(z_au)(\star \otimes \star)\decon(z_bv) ={}&{} 1\otimes (z_au\star z_bv) + (z_a\otimes 1)\left(\decon(u)\bstar\decon(z_bv)\right)\\
&+ (z_b\otimes 1)\left(\decon(z_au)\bstar\decon(v)\right)\\
&-(L_{z_a\lozenge z_b}\otimes 1)\left(\decon(u)\bstar\decon(v)\right)\\
={}&{} \decon(z_au\star z_bv).
\end{align*}
Thus $\decon$ is a $\star$-homomorphism.

Finally it remains to confirm $S$ satisfies the antipode axioms. We will only show that 
\[\sum_{k=0}^r S(z_{i_1}\ldots z_{i_k})\star z_{i_{k+1}}\ldots z_{i_r}=0.\]
In particular, we need to show that
\begin{align*}
\sum_{k=0}^r (-1)^k (z_{i_k}\ldots z_{i_1})\star (z_{i_{k+1}}\ldots z_{i_r}) =&{} \sum_{k=0}^r (-1)^k z_{i_k}(z_{i_{k-1}}\ldots z_{i_1}\star z_{i_{k+1}}\ldots z_{i_r})\\
&+(-1)^kz_{i_{k+1}}(z_{i_k}\ldots z_{i_1}\star z_{i_{k+2}}\ldots z_{i_r})\\
&+(-1)^{k+1}L_{z_{i_k}\lozenge z_{i_k+1}}(z_{i_{k-1}}\ldots z_{i_1}\star z_{i_{k+2}}\ldots z_{i_r})
=&{} 0.
\end{align*}
For fixed $k$, the part of this sum begining with $z_{i_k}$  is given by
\[ (-1)^kz_{i_k}(z_{i_{k-1}}\ldots z_{i_1}\star z_{i_{k+1}}\ldots z_{i_r}) + (-1)^{k-1} z_{i_k}(z_{i_{k-1}}\ldots z_{i_{k+1}}\ldots z_{i_r})=0.\]
Thus the only surviving terms are 
\[\sum_{k=0}^r (-1)^{k+1}L_{z_{i_k}\lozenge z_{i_{k+1}}}(z_{i_{k-1}}\ldots z_{i_1}\star z_{i_{k+2}}\ldots z_{i_r}).\]
We claim this vanishes, and argue by induction on the length. One may easily verify that these terms vanish for $r=0,1,2,3$. Suppose these terms vanish for any work of length less than $r$. Then, using the recursive forumula, we have that
\begin{align*}
&\sum_{k=0}^r (-1)^{k+1}L_{z_{i_k}\lozenge z_{i_{k+1}}}(z_{i_{k-1}}\ldots z_{i_1}\star z_{i_{k+2}}\ldots z_{i_r})\\
=&{} \sum_{k=0}^r (-1)^{k+1}(z_{i_{k-1}}\lozenge z_{i_k}\lozenge z_{i_{k+1}})(z_{i_{k-2}}\ldots z_{i_1}\star z_{i_{k+2}}\ldots z_{i_r})\\
&+(-1)^{k+1}(z_{i_{k}}\lozenge z_{i_{k+1}}\lozenge z_{i_{k+2}})(z_{i_{k-1}}\ldots z_{i_1}\star z_{i_{k+3}}\ldots z_{i_r})\\
&+(-1)^k L_{z_{i_{k-1}}\lozenge z_{i_k}\lozenge z_{i_{k+1}}\lozenge z_{i_{k+2}}}(z_{i_{k-2}}\ldots z_{i_1}\star z_{i_{k+3}}\ldots z_{i_r})
\end{align*}
Clearly, the first two lines cancel out. Iterating the process, we obtain the desired equality.
\end{proof}

\begin{thm}
The algebra morphism $\Psi_{\tanh}:(\Q\langle Z\rangle,\sh)\to (\Q\langle Z\rangle, \star)$ is a Hopf algebra isomorphism of $(\Q\langle Z\rangle,\sh,\decon)\to (\Q\langle Z\rangle, \star,\decon)$.
\end{thm}
\begin{proof}
From Proposition \ref{tanhmorphism}, we have that $\Psi_{\tanh}$ is an isomorphism of algebras. As a bialgebra admits at most one Hopf algebra structure, it suffices to show that $\Psi_{\tanh}$ is compatible with the counit and coproduct. It is immediate that $\Psi_{\tanh}\circ \epsilon = \epsilon\circ \Psi_{\tanh}$, and so it remains to show that $\decon\circ \Psi_{\tanh} = (\Psi_{\tanh}\otimes\Psi_{\tanh})\circ \decon$. This is trivial for the empty word, and for a non-empty word $w=z_{i_1}\ldots z_{i_r}$ in $Z$, we have $\decon(\Psi_{\tanh}(w))$ is given by
\begin{align*}
&{} \sum_{(j_1,\ldots,j_l)\in\mathcal{C}(r)} c_{j_1}\ldots c_{j_l} \sum_{k=0}^l (j_1,\ldots,j_k)[z_{i_1}\ldots z_{i_{j_1+\cdots + j_k}}]\otimes (j_{k+1},\ldots, j_{l})[z_{i_{j_1+\cdots + j_k+1}}\ldots z_{i_r}]\\
=&{} \sum_{p+q=r} \sum_{\substack{(j_1,\ldots,j_m)\in\mathcal{C}(p)\\(k_1,\ldots,k_n)\in\mathcal{C}(q)}} c_{j_1}\ldots c_{j_m}c_{k_1}\ldots c_{k_n} (j_1,\ldots, j_m)[z_{i_1}\ldots z_{i_p}] \otimes (k_1,\ldots, k_n)[z_{i_{p+1}},\ldots,z_{i_r}]\\
=&{} \sum_{p+q=r} \Psi_{\tanh}(z_{i_1}\ldots z_{i_p})\otimes \Psi_{\tanh}(z_{i_{p+1}}\ldots z_{i_r})\\
={}&{} (\Psi_{\tanh}\otimes\Psi_{\tanh})(\decon(w))
\end{align*}
where $c_n$ is the coefficient of $x^n$ in the Taylor series of $\tanh(x)$.
\end{proof}

\begin{remark}
This proof also shows that the graded duals introduced in the proof of Proposition \ref{tanhmorphism} are isomorphic as Hopf algebras when equipped with the concatenation product $\cdot$.
\[(\Q\langle Z\rangle,\cdot,\Delta)\cong (\Q\langle Z\rangle, \cdot,\bDelta).\]
In particular, this allows us to computed the primitive elements of  $(\Q\langle Z\rangle, \cdot,\bDelta)$ in terms of primitive elements of $(\Q\langle Z\rangle, \cdot,\Delta)$. 

Knowing \ref{fullblockshuff}, we must have that (motivic) multiple zeta values satisfy relations coming from the block shuffle algebra, and hence the motivic Lie algebra $\gmot$ \cite{delmixedtate} encoding relations among motivic multiple zeta values \cite{brownmixedtate} should inject into the set of primitive elements of $(\Q\langle Z\rangle, \cdot,\bDelta)$. It is known from work following the author's thesis \cite{keilthyblock1}\cite{mythesis} that a graded version $\bg$ of the motivic Lie algebra injects into the set of primitive elements of $(\Q\langle Z\rangle, \cdot,\Delta)$.

As such, one might hope to find maps such that the following diagram commutes
\begin{center}
\begin{tikzcd}
\bg \arrow[r] \arrow[d, hook]& \gmot \arrow[d, hook]\\
(\Q\langle Z\rangle,\Delta) \arrow[r, "\Psi_{\tanh}"] & (\Q\langle Z\rangle,\bDelta)
\end{tikzcd}
\end{center}

Both $\gmot$ and $\bg$ are known to be isomorphic to the free Lie algebra $\text{Lie}[\sigma_3,\sigma_5,\sigma_7,\ldots]$, but this isomorphism is only canonical in the case of $\bg$. An explicit, canonical commutative diagram, as above, would allow us to somehow ``lift'' the isomorphism 
\[\bg \cong \text{Lie}[\sigma_3,\sigma_5,\sigma_7,\ldots]\]
to an explicit canonical choice of isomorphism
\[\gmot \cong \text{Lie}[\sigma_3,\sigma_5,\sigma_7,\ldots]\]
thereby enabling us to construct examples of non-trivial rational associators \cite{zeta3}, to establish relations via inductive methods like those of \cite{keilthyblock1}, and explore explicitly the ``extra generators'' required to correct the depth defect \cite{depthgraded}.

\end{remark}

\bibliographystyle{abbrv}
\bibliography{summary_bib}
\end{document}